\documentclass{amsart}
\usepackage{amssymb,latexsym}
\usepackage[mathscr]{eucal}
\usepackage{mathtools}
\usepackage{enumitem}

\theoremstyle{plain}
\newtheorem{theorem}{Theorem}
\newtheorem*{ML}{Main Lemma}
\newtheorem{lemma}{Lemma}%[theorem]

\theoremstyle{definition}
\newtheorem{example}{Example}

\theoremstyle{remark} 
\newtheorem*{remark}{Remark}

\newcommand{\refT}[1]{Theorem~\ref{T#1}}

\newcommand{\refL}[1]{Lemma~\ref{L#1}}

\newcommand{\ga}{\alpha}  \newcommand{\gb}{\beta}  
\newcommand{\gd}{\delta}  \newcommand{\dn}{\delta_n}
\newcommand{\gD}{\Delta}    
\newcommand{\gf}{\varphi}  \newcommand{\vta}{\vartheta}

\newcommand{\Z}{\mathbb{Z}}  
\newcommand{\CZ}{\mathcal{Z}}  

\newcommand{\cycp}{\Phi_{pq}}  \newcommand{\Q}{Q_{\{p,q\}}}
\newcommand{\la}{\langle}  \newcommand{\ra}{\rangle}
\newcommand{\vz}{\mathbf{z}}  \newcommand{\tz}{(z_1,z_2,\dots,z_t)}
\newcommand{\ib}{\bar{\iota}}

\author[G. Bachman]{Gennady Bachman}
\address{Department of Mathematical Sciences\\ University of Nevada Las Vegas\\
4505 Maryland Parkway, Box 454020\\
Las Vegas, Nevada 89154-4020, USA}
\email{gennady.bachman@unlv.edu}

\date{May 2026}

\begin{document}

\title[Gaps of Binary Semigroups]{Gaps of Binary Numerical Semigroups and of Binary Inclusion-Exclusion Polynomials}

\begin{abstract}
Let $p$ be a given modulus, let $u$ be prime to $p$, and consider the linear permutation $u\cdot n\pmod p$ of the residue system modulo $p$. Writing $\la x\ra_p$ to denote the least nonnegative residue of $x$ modulo $p$, we say that a pair of integers $(a,b)$ is a dominant pair of this permutation if either the inequality $\max(\la ua\ra_p,\la ub\ra_p)<\min_{a<n<b}\la un\ra_p$, or the inequality $\min(\la ua\ra_p,\la ub\ra_p)>\max_{a<n<b}\la un\ra_p$ hold. The main technical part of this work gives analysis of this property of linear permutations of residue systems. We then apply this analysis to the problems that motivated it, and give (i) complete description of the gapsets of binary inclusion-exclusion polynomials $Q_{\{p,q\}}$ (which include binary cyclotomic polynomials $\Phi_{pq}$ as its principal special case), and (ii) complete description of all possible distances between consecutive elements of a numerical semigroup $\la p,q\ra$.
\end{abstract}

  %% Primary 11B83; Secondary 11C08


\maketitle

\section{Introduction}
In their recent work \cite{HLLP}, Hong, et al. introduced a new and interesting puzzle about the behavior of coefficients of cyclotomic polynomials $\Phi_{pq}(x)$, where $p<q$ are distinct odd primes. To state their problem, we need to introduce the notion of the \emph{gapset} of a polynomial. Given a polynomial $f(x)$ of degree $d$, write it down omitting its zero coefficients $f(x)=c_1x^{e_1}+c_2x^{e_2}+\dots+c_kx^{e_k}$, where $0\le e_1<e_2<\dots<e_k=d$ and $c_i\ne0$. Then, if $k>1$, the gapset of $f$ is given by $G(f)\coloneq\{e_{i+1}-e_i\}$, and the individual elements are the \emph{gaps} of $f$. Motivated by a certain application, Hong, et al, were particularly interested in the largest gap of $\Phi_{pq}$, denoted by $g_1(\Phi_{pq})$. They showed that $g_1(\Phi_{pq})=p-1$, and in the process initiated the problem of studying the gaps of $\Phi_{pq}$.

This was followed by several new contributions investigating this and related problems. Below we describe those among them that are most relevant to the objectives of this paper. For a more comprehensive discussion and references to these contributions we refer the reader to the papers \cite{AAHL} and \cite{Sa}. New proofs of the $g_1(\Phi_{pq})=p-1$ result were given in \cite{Mo, Zh, CCLMS}. Zhang \cite{Zh} showed that the gap $p-1$ occurs exactly $\lfloor q/p\rfloor$ times. The simplest special case of the gaps problem for $\Phi_{pq}$ occurs when $q\equiv\pm1\pmod p$. The analysis of this case was given, among other results, by Camburu, et al, in \cite{CCLMS}. They showed that the relation $q\equiv\pm1\pmod p$ is equivalent to the evaluation $G(\Phi_{pq})=\{1,2,\dots,p-1\}$. In addition to the largest gap, we now also know the second largest gap of $\Phi_{pq}$, denoted by $g_2(\Phi_{pq})$. This result is due to Cafure and Cesaratto \cite{CC} and it says that $g_2(\Phi_{pq})=\max(r-1,p-r-1)$, where $r$ is the remainder of $q$ on division by $p$. Finally, we mention what is known about the gaps of general cyclotomic polynomials $\Phi_n$. It is clear the the situation here is much more challenging than in the binary case we have been discussing above. Nevertheless, Al-Kateeb, et al, \cite{AAHL} were able to determine the largest gap of $\Phi_n$ for $n$ satisfying certain constraints. First, they note that if $n_0=\prod_{p\mid n, p\ge3}p$ is an ``odd squarefree kernel'' of $n$, then $g_1(\Phi_n)=\frac n{n_0}g_1(\Phi_{n_0})$. (Note that this applies to our case $n_0=pq$.) So the problem reduces to the study of gaps of $\Phi_n$ for squarefree odd integers $n$. For such $n$, they showed that if the largest prime factor $q$ of $n$ satisfies $q>\sqrt n$, then $g_1(\Phi_n)=\gf(n/q)$, where $\gf$ is the Euler's totient function.

Henceforth we shall assume that integers $q>p\ge3$ are coprime, but not necessarily prime numbers. We will work in the more general setting of binary inclusion-exclusion polynomials which are defined as quotients
\begin{equation}\label{1}
\Q(x)\coloneq\frac{(1-x^{pq})(1-x)}{(1-x^p)(1-x^q)}.
\end{equation}
It is easy to verify (see \cite{Ba} for an introduction to inclusion-exclusion polynomials) that $\Q(x)$ reduces to a polynomial of degree $\gf(p,q)\coloneq(p-1)(q-1)$, and we write
\[
\Q(x)=\sum_{0\le m\le\gf(p,q)}a_mx^m \qquad[a_m=a_m(p,q)].
\]
When the pair $\{p,q\}$ is a pair of prime numbers, polynomial $\Q$ is better known as the cyclotomic polynomial $\cycp$. In particular, our gap results for this larger class of polynomials apply to the subclass of cyclotomic polynomials $\cycp$.

From the definition \eqref{1} and the fact that the degree of $\Q$ is $\gf(p,q)$, it follows that
\begin{equation}\label{2}
\Q(x)\equiv(1-x)\sum_{i,j\ge0}x^{iq+jp} \pmod{x^{\gf(p,q)+1}}.
\end{equation}
Two things are now clear. First, is that each coefficient $a_m$ takes on one of three possible values, 0 or $\pm1$. Second, is that it is all about integers representable as linear combinations of $p$ and $q$ with nonnegative integer coefficients,
\begin{equation}\label{2.a}
\la p,q\ra\coloneq\{\,iq+jp\mid i,j\ge0\,\}.
\end{equation}
Such a structure is called a numerical semigroup of embedding dimension two (see, for example \cite[\S7]{RA})---we took the liberty of abbreviating this to ``binary semigroup'' in our title. In view of \eqref{2}, of crucial importance to us is the interesting part of $\la p,q\ra$, namely
\[
S(p,q)\coloneq\la p,q\ra\cap[0,(p-1)(q-1)].
\]
We say interesting, because the largest integer not in $\la p,q\ra$ is $pq-p-q$ (see \cite[\S 2.1]{RA}). The total number of integers in $[0,(p-1)(q-1)]$ missing from $S(p,q)$ is $\vta\coloneq\gf(p,q)/2$ (see \cite[\S 5.1]{RA}), the largest of which is $pq-p-q$, and the number of integers in $S(p,q)$ is $\vta+1$. We ought to comment that, unfortunately, we encounter here a bit of a notational conflict. The terms ``gap'' as well as ``nongap'' are also used in the context of these semigroups---the elements of $\la p,q\ra$ are the nongaps and the missing integers are the gaps. We avoid this conflict by using the terms representable and nonrepresentable instead.

It readily follows from \eqref{2} that the gaps of $\Q$ correspond to blocks of consecutive representable integers and blocks of consecutive nonrepresentable integers in $[0,\gf(p,q)]$. More precisely (see Section 4 for details), by a block we mean two or more of consecutive integers of the same type, and the corresponding gap equals to the length of the block. Let us list the elements of $S(p,q)$ in the increasing order: $\ell_0<\ell_1<\dots<\ell_{\vta}$, so that $\ell_0=0,\,\ell_1=p,\,\dots,\,\ell_{\vta}=(p-1)(q-1)$, and do the same for nonrepresentables: $n_1<n_2<\dots<n_{\vta}$, so that $n_1=1,\, n_2=2,\,\dots,\,n_{\vta}=pq-p-q$. Then block lengths are given by various values of $\ell_{j+1}-\ell_j-1>0$ and $n_{j+1}-n_j-1>0$, and the preceding assertion is that
\[
G(\Q)=\{\,\ell_{j+1}-\ell_j-1>0\,\}\cup\{\,n_{j+1}-n_j-1>0\,\}.
\]
So the problem of understanding the gapset of polynomials $\Q$ comes to the same thing as the problem of understanding the sets
\begin{equation}\label{3}
S_\gD(p,q)\coloneq\{\ell_{j+1}-\ell_j\} \quad\text{and}\quad N_\gD(p,q)\coloneq\{n_{j+1}-n_j\},
\end{equation}
an interesting problem in its own right. We give a complete solution in Section 3.

Now consider $p$ as a given modulus and let $u$ be a fixed integer coprime with $p$. (In our applications to $\Q$ and to $\la p,q\ra$, $u$ will be a multiplicative inverse of $q$ modulo $p$.) Linear congruence $u\cdot n\pmod p$ gives a permutation of the residue system modulo $p$. We shall show that, for the problems discussed above, the crux of the matter is to understand the following property of such permutations. We use the notation $\la x\ra_p$ to denote the least nonnegative residue of $x$ modulo $p$, that is,
\[
\la x\ra_p\equiv x\pmod p \quad\text{and}\quad 0\le\la x\ra_p<p.
\]
We wish to understand the \emph{dominant pairs} of integers $a$ and $b$: these are integers $a<b$ such that either
\begin{equation}\label{4}
\max(\la ua\ra_p, \la ub\ra_p)<\min_{a<n<b}\la un\ra_p
\end{equation}
or
\begin{equation}\label{5}
\min(\la ua\ra_p, \la ub\ra_p)>\max_{a<n<b}\la un\ra_p.
\end{equation}
To give the reader a quick idea of the relevance of such pairs, we point out that it is plain that the maximum difference $b-a$ for pairs $(a,b)$ satisfying \eqref{4} is $p$, for example $a=0$ and $b=p$, and the maximum gap of $\Q$ is $p-1=b-a-1$.

The next section, the main part of this paper, studies \eqref{4} and \eqref{5} and, in particular, gives the complete description of the set of values $\{b-a\}$ for the dominant pairs. To carry out our analysis we use what we shall call the \emph{Euclidean algorithm for residue systems}. This gives a representation of a residue system modulo $p$ that renders conditions \eqref{4} and \eqref{5} quite transparent. We hope that this representation is of independent interest and might see other uses in the future.

\section{Linear Permutations}
Let a modulus $p$ ($\ge3$) and an integer $u$ prime to $p$ be given. In this section we study the dominant pairs $(a,b)$, that is, the solutions of the inequalities \eqref{4} and \eqref{5}. Let $r_1$, $0<r_1<p$, be the multiplicative inverse of $u$ modulo $p$, so that $\la ur_1\ra_p=1$.

\begin{lemma}\label{L1}
Let $D=D(p,u)$ be the collection of all pairs $(a,b)$ satisfying \eqref{4} and let $D'=D'(p,u)$ be the collection of all pairs $(a,b)$ satisfying \eqref{5}. Then the involution
\[
(x,y) \longmapsto (-r_1-y, -r_1-x)
\]
is a bijection on the sets $D$ and $D'$, whence
\[
\{\, b-a \mid (a,b)\in D \,\}=\{\, b-a \mid (a,b)\in D' \,\}.
\]
\end{lemma}

\begin{proof}
Observe that $p-1-\la ux\ra_p=\la u(-r_1-x)\ra_p$, since $0\le p-1-\la ux\ra_p\le p-1$. It follows that if \eqref{4} holds for the pair $(a,b)$, then \eqref{5} holds with $(a,b)$ replaced by $(-r_1-b,-r_1-a)$. In the opposite direction, if \eqref{5} holds for the pair $(a,b)$, then \eqref{4} holds with $(a,b)$ replaced by $(-r_1-b,-r_1-a)$.
\end{proof}

In view of \refL{1}, it suffices to examine the inequality \eqref{4}. The key objective of this section is to show that this inequality is governed by the outcome of the Euclidean algorithm applied to integers $p$ and $r_1$. We will find it convenient to use the following slightly nonstandard notation for the Euclidean algorithm. Put $r_0=p$ and define the sequences $\bigl(r_i\bigr)_{i=0}^{t+1}$ and $\bigl(Z_i\bigr)_{i=1}^t$ by the formula
\begin{equation}\label{2.1}
r_{i-1}=Z_ir_i+r_{i+1}, \quad 0\le r_{i+1}<r_i,
\end{equation}
for $1\le i\le t$, where $r_t=1$, so that $r_{t+1}=0$ and $Z_t=r_{t-1}$.

\begin{ML}
Let $D_\gD=D_\gD(p,u)\coloneq\{\,b-a\mid(a,b)\in D\,\}$. Then
\[
D_\gD=\{\,r_{i-1}-zr_i\mid 0\le z\le Z_i-1 \text{ and } 1\le i\le t\,\}.
\]
\end{ML}

\begin{remark}
In view of the process that generates solutions to \eqref{4}, $D_\gD(p,u)$ naturally partitions into $t$ parts $D_i(p,u)$:
\begin{equation}\label{2.2}
D_\gD=\bigcup_{i=1}^tD_i=\bigcup_{i=1}^t\{\,r_{i-1}-zr_i\mid 0\le z\le Z_i-1\,\}.
\end{equation}
In this scheme the ``larger differences'' go into $D_1$, ``next larger differences'' go into $D_2$, etc.
\end{remark}

Our proof of the Main Lemma uses a representation of the residue system modulo $p$ that is a byproduct of the Euclidean algorithm for $p$ and $r_1$. We carry out this construction next.

\subsection{The Euclidean Algorithm for Residue Systems}
\begin{lemma}\label{L2}
For $1\le i\le t$, write $t=i+2k+1$ or $t=i+2k$. Then we have
\[
\sum_{0\le j\le k}Z_{i+2j}r_{i+2j}=\begin{cases}
r_{i-1}-1, &\text{if } i=t-2k-1\\
r_{i-1}, &\text{if } i=t-2k.
\end{cases} \]
\end{lemma}

\begin{proof}
Recall that, by \eqref{2.1}, $Z_ir_i=r_{i-1}-r_{i+1}$. Therefore,
\[
\sum_{0\le j\le k}Z_{i+2j}r_{i+2j}=\sum_{0\le j\le k}(r_{i-1+2j}-r_{i+1+2j})=r_{i-1}-r_{i+1+2k},
\]
and the claim follows since $r_t=1$ and $r_{t+1}=0$.
\end{proof}

Next, we introduce several conventions and definitions. For each $1\le i\le t$, variable $z_i$ will always denote an integer in the range $0\le z_i\le Z_i$. The following generalizations of the sum appearing in \refL{2} will prove to be useful. Put, for $0\le j_1\le j_2\le(t-1)/2$,
\begin{equation}\label{2.3}
\sigma_1(j_1,j_2;z_{2j_2+1})\coloneq\sum_{j_1\le j\le j_2-1}Z_{2j+1}r_{2j+1}+z_{2j_2+1}r_{2j_2+1},
\end{equation}
note that this is just $z_{2j_2+1}r_{2j_2+1}$ if $j_2=j_1$, and, for $0\le j_1\le j_2\le t/2$,
\begin{equation}\label{2.4}
\sigma_0(j_1,j_2;z_{2j_2})\coloneq\sum_{j_1\le j\le j_2-1}Z_{2j}r_{2j}+z_{2j_2}r_{2j_2}.
\end{equation}
We let $\vz$ denote a $t$-tuple $\tz$. Not all such $t$-tuples will play a role in our construction. In addition to $\mathbf{0}=(0,0,\dots,0)$, we will use $\vz$ with the properties: (i) the smallest index $i$ such that $z_i>0$ is odd, and (ii) $0\le z_t\le Z_t-1$. Let us denote the subset of such $t$-tuples by $\CZ$, and let $\ib=\ib(\vz)$ denote the special index of property (i). Finally, for $\vz\in\CZ$, we put
\begin{equation}\label{2.5}
R(\vz)\coloneq\sum_{1\le i\le t}z_i(-1)^{i-1}r_i.
\end{equation}

\begin{lemma}\label{L3}
For $\vz\in\CZ$, we have $0<R(\vz)<p$, unless $\vz=\textbf{0}$, and $R(\textbf{0})=0$.
\end{lemma}

\begin{proof}
For the upper bound, we have
\[
R(\vz)\le z_1r_1+z_3r_3+\dots\le p-1,
\]
by \refL{2} with $i=1$. Note that we use the fact that $z_t\le Z_t-1$, to get the last inequality for odd values of $t$. For the lower bound, since $z_i=0$ for all $i<\ib$, we have
\[
R(\vz)\ge z_{\ib}r_{\ib}-(z_{\ib+1}r_{\ib+1}+z_{\ib+3}r_{\ib+3}+\dots).
\]
Another application of \refL{2} (adjusted for $z_t\le Z_t-1$), with $i=\ib+1$, gives $R(\vz)\ge1$.
\end{proof}

\begin{lemma}\label{L4}
Each integer $n\in[0,p)$ has a representation in the form $n=R(\vz)$, for some $\vz\in\CZ$.
\end{lemma}

\begin{remark}
The example $n=r_3=r_1-Z_2r_2$ shows that a representation $n=R(\vz)$ need not be unique. This ambiguity is of no moment in the present paper. However, it might be worth noting that the algorithm given in the proof of the lemma does give an injection $n\mapsto\vz(n)$ from $[0,p)$ to $\CZ$.
\end{remark}

\begin{proof}
$0=R(\mathbf{0})$. Now fix an arbitrary integer $0<n<p$. By \refL{2} with $i=1$ and definition \eqref{2.3}, there is the smallest index $i=2k_1+1$ such that $\sigma_1(0,k_1;Z_{2k_1+1})\ge n$. Let $z_{2k_1+1}$ be the smallest value (it must be $\ge1$) such that $\sigma_1(0,k_1;z_{2k_1+1})\ge n$. If this holds with equality, we stop. If not, observe that
\[
0<n_1=\sigma_1(0,k_1;z_{2k_1+1})-n<r_{2k_1+1},
\]
whence $2k_1+1<t$. Another application of \refL{2}, this time with $i=2k_1+2$, gives, using \eqref{2.4}, $\sigma_0(k_1+1,k_2;z_{2k_2})\ge n_1$, where the index $i=2k_2$ and the value $z_{2k_2}\ (\ge1)$ are chosen to be minimal, so that $\sigma_0(k_1+1,k_2;z_{2k_2})<n_1+r_{2k_2}$. In the case that the equality holds, we get
\[
n=\sigma_1(0,k_1;z_{2k_1+1})-\sigma_0(k_1+1,k_2;z_{2k_2})
\]
and we stop. Otherwise, we have
\begin{equation}\label{A1}
\sigma_1(0,k_1;z_{2k_1+1})-\sigma_0(k_1+1,k_2;z_{2k_2})<n<\sigma_1(0,k_1;z_{2k_1+1}),
\end{equation}
and
\[
n_2=n-\sigma_1(0,k_1;z_{2k_1+1})+\sigma_0(k_1+1,k_2;z_{2k_2})<r_{2k_2} \quad\text{and}\quad 2k_2<t.
\]

Our proof strategy and technique are now clear. Thus the next step in our argument is to further zoom in on $n$ from above and improve the upper bound in \eqref{A1} to
\[  %\sigma_1(0,k_1;z_{2k_1+1})-\sigma_0(k_1+1,k_2;z_{2k_2})<
n\le\sigma_1(0,k_1;z_{2k_1+1})-\sigma_0(k_1+1,k_2;z_{2k_2})+\sigma_1(k_2,k_3;z_{2k_3+1}),
\]
and note that if this upper bound is strict, then we have
\[
0<\sigma_1(0,k_1;z_{2k_1+1})-\sigma_0(k_1+1,k_2;z_{2k_2})+\sigma_1(k_2,k_3;z_{2k_3+1})-n<r_{2k_3+1},
\]
and $2k_3+1<t$. It is also clear that this zooming in algorithm must terminate after a finite number of steps, and that the termination means that we arrive, after $s$ steps, at a representation
\begin{equation}\label{2.6}
n=\sigma_1(0,k_1;z_{2k_1+1})-\sigma_0(k_1+1,k_2;z_{2k_2})+-\dots,
\end{equation}
where the last term is either $+\sigma_1(k_{s-1},k_s;z_{2k_s+1})$, if $s$ is odd, or $-\sigma_0(k_{s-1}+1,k_s;z_{2k_s})$, if $s$ is even.

Because there is no overlapping in the arguments of any of the $\sigma$s on the right side of \eqref{2.6}, it is now just a formality to show that this sum is, in fact, $R(\vz)$. To that end, we define $\vz_1, \vz_2, \dots, \vz_s$ as:
\[
\vz_1=(Z_1,0,Z_3,0,\dots,z_{2k_1+1},\textrm{O}),
\]
where O denotes a block of 0s of the appropriate size and $z_{2k_1+1}$ is the first entry if $k_1=0$,
\[
\vz_2=(\textrm{O},Z_{2k_1+2},0,Z_{2k_1+4},0,\dots,z_{2k_2},\textrm{O}), \
\vz_3=(\textrm{O},Z_{2k_2+1},0,\dots,z_{2k_3+1},\textrm{O}),
\]
etc, and put $\vz=\vz_1+\vz_2+\dots+\vz_s$. By the definition \eqref{2.5} of $R$, it is now plain that the right side of \eqref{2.6} equals to $R(\vz)$, and the proof is complete.
\end{proof}

Next, we show how a representation $n=R(\vz)$ will be used to evaluate $\la un\ra_p$. We begin by defining the sequence $s_0=0,\ s_1=1$, and, for $1\le i\le t$,
\begin{equation}\label{2.7}
s_{i+1}=Z_is_i+s_{i-1}.
\end{equation}

\begin{lemma}\label{L5}
The identity $p=s_ir_{i-1}+s_{i-1}r_i$ holds for all $1\le i\le t+1$. In particular, we conclude that $s_{t+1}=p$, since $r_t=1$ and $r_{t+1}=0$, and that $s_i$ is strictly increasing to $p$.
\end{lemma}

\begin{proof}
By \eqref{2.1} and \eqref{2.7}, we have
\begin{align*}
p&=1\cdot r_0+0\cdot r_1=s_1r_0+s_0r_1=s_1(Z_1r_1+r_2)+s_0r_1\\
&=(Z_1s_1+s_0)r_1+s_1r_2=s_2r_1+s_1r_2.
\end{align*}
Proceeding inductively:
\begin{align*}
p&=s_ir_{i-1}+s_{i-1}r_i=s_i(Z_ir_i+r_{i+1})+s_{i-1}r_i\\
&=(Z_is_i+s_{i-1})r_i+s_ir_{i+1}=s_{i+1}r_i+s_ir_{i+1},
\end{align*}
completes the proof.
\end{proof}

\begin{lemma}\label{L6}
$\la u(-1)^{i-1}r_i\ra_p=s_i$, for $0\le i\le t$.
\end{lemma}

\begin{proof}
To start an induction argument, we verify that $\la ur_0\ra_p=0=s_0$, $\la ur_1\ra_p=1=s_1$, and, by \eqref{2.1} and \eqref{2.7},
\[
\la u(-r_2)\ra_p=\la uZ_1r_1\ra_p=Z_1=s_2.
\]
We complete the proof by induction. By \refL{5}, for $2\le i\le t$, we have
\[
s_ir_{i-1}+s_{i-1}r_i=p \quad\text{and}\quad s_ir_{i-1},\, s_{i-1}r_i>0.
\]
Therefore, $\la s_iur_{i-1}\ra_p+\la s_{i-1}ur_i\ra_p=p$. To apply the induction hypothesis, we must distinguish two cases according to the parity of $i$. If $i$ is odd, we get
\[
\la s_{i-1}ur_i\ra_p=p-\la s_iur_{i-1}\ra_p=\la s_iu(-r_{i-1})\ra_p=\la s_is_{i-1}\ra_p,
\]
whence $\la ur_i\ra_p=s_i$. Similarly, if $i$ is even, we get
\[
p-\la s_{i-1}ur_i\ra_p=\la s_{i-1}u(-r_i)\ra_p=\la s_iur_{i-1}\ra_p=\la s_is_{i-1}\ra_p,
\]
whence $\la u(-r_i)\ra_p=s_i$, and the claim follows.
\end{proof}

We are now ready to evaluate $\la un\ra_p$ in terms of the representation $n=R(\vz)$.

\begin{lemma}\label{L7}
For $\vz\in\CZ$, we have 
\[
\la uR(\vz)\ra_p=\la u\sum_{1\le i\le t}z_i(-1)^{i-1}r_i\ra_p=\sum_{1\le i\le t}z_is_i.
\]
\end{lemma}

\begin{proof}
Congruence $uR(\vz)\equiv\sum_{1\le i\le t}z_is_i \pmod p$ follows by \refL{6}, and it only remains to show that $\sum_{1\le i\le t}z_is_i<p$. Recall that $z_t\le Z_t-1$, so that we have
\begin{align*}
\sum_{1\le i\le t}z_is_i &\le \sum_{1\le i\le t}Z_is_i-s_t=\sum_{1\le i\le t}(s_{i+1}-s_{i-1})-s_t\\
 &=s_{t+1}-s_1-s_0=p-1,
\end{align*}
by \eqref{2.7} and \refL{5}, and the claim follows.
\end{proof}

\subsection{Proof of the Main Lemma}
We begin the proof by getting the trivial cases out of the way, primarily because it helps to avoid some notational clutter. The inequality \eqref{4} is vacuously satisfied by every pair of consecutive integers $a$ and $b=a+1$, so $1\in D_\gD$. Recall that $Z_t=r_{t-1}$, by \eqref{2.1}, so that $1=r_{t-1}-(Z_t-1)r_t$, as required. Another obvious solution of \eqref{4} is the pair $(0,p)$, so $p=r_0-0\cdot r_1\in D_\gD$. Note also that no $(a,b)\in D$ can have a multiple of $p$ in the open interval $(a,b)$, so we certainly must have $b-a\le p$. It is also plain that the inequality \eqref{4} is preserved under the translation of the closed interval $[a,b]$ by a multiple of $p$. So, in what follows, we determine the remaining elements of $D_\gD$ by considering solutions $(a,b)$ of \eqref{4} satisfying
\begin{equation}\label{2.8}
0\le a<b\le p \quad\text{and}\quad 1<b-a<p.
\end{equation}

We carry out the remainder of this proof in two parts. First, we show that every $r_{i-1}-zr_i$ specified in the lemma is an element of $D_\gD$. This part is carried out in Lemmas 8 and 9 below. Then, in Lemmas 10 and 11, we show that every solution $(a,b)$ od \eqref{4} satisfies $b-a=r_{i-1}-zr_i$, as claimed.

\begin{lemma}\label{L8}
Let $a=p-r_{2k}+zr_{2k+1}$, where $2k+1\le t$ and $0\le z\le Z_{2k+1}-1$. Then, we have
\[
\la ua\ra_p<\min_{a<n<p}\la un\ra_p,
\]
whence $r_{2k}-zr_{2k+1}\in D_\gD$.
\end{lemma}

\begin{proof}
The special case where $2k=0$ is trivial and it is helpful to view it separately. In this case $a=zr_1$ (possibly $=0$), $\la uzr_1\ra_p=z$ and, plainly, $\la un\ra_p>z$ for $zr_1<n<p$. And now for the general case. First, we observe that
\[
\la ua\ra_p=\la u(-r_{2k}+zr_{2k+1})\ra_p=s_{2k}+zs_{2k+1}<(z+1)s_{2k+1},
\]
by Lemmas 5, 6 and \eqref{2.7}. To complete the proof, we show that if $a<n<p$, then $\la un\ra_p\ge(z+1)s_{2k+1}$.

Fix $n$ in $(a,p)$. By \refL{4}, $n=R(\vz)$ for some $\vz=(z_1,\dots,z_t)\in\CZ$. If $z_i>0$ for some $i\ge2k+2$, then, by \refL{7} and \eqref{2.7},
\[
\la un\ra_p\ge z_is_i\ge s_{2k+2}\ge(z+1)s_{2k+1},
\]
as required. (Note that the last inequality could hold with equality in the trivial case $2k=0$, if $z=Z_1-1$.) Now assume that either $z_i=0$ for all $i\ge2k+2$ or that $2k+1=t$. Then
\[
n=R(\vz)\le\sum_{0\le j\le k-1}Z_{2j+1}r_{2j+1}+z_{2k+1}r_{2k+1}.
\]
Observe that \refL{2} implies that $\sum_{0\le j\le k-1}Z_{2j+1}r_{2j+1}=p-r_{2k}$, whence $n\le p-r_{2k}+z_{2k+1}r_{2k+1}$. This means that $z_{2k+1}$ must be $>z$ and \refL{7} gives
\[
\la un\ra_p\ge z_{2k+1}s_{2k+1}\ge(z+1)s_{2k+1}.
\]
\end{proof}

\begin{lemma}\label{L9}
Let $b=r_{2k-1}-zr_{2k}$, where $2k\le t$ and $0\le z\le Z_{2k}-1$. Then, we have
\[
\la ub\ra_p<\min_{0<n<b}\la un\ra_p,
\]
whence $r_{2k-1}-zr_{2k}\in D_\gD$.
\end{lemma}

\begin{proof}
The proof of this lemma closely mirrors the proof of \refL{8} and we move more swiftly. We have
\[
\la ub\ra_p=s_{2k-1}+zs_{2k}<(z+1)s_{2k},
\]
and we complete the proof by showing that $\la un\ra_p>(z+1)s_{2k}$, for $0<n<b$. Fix $n$ in this range and let $\vz=(z_i)\in\CZ$ be such that $n=R(\vz)$. If $z_i>0$ for some $i\ge2k+1$, then 
\[
\la un\ra_p\ge z_is_i\ge s_{2k+1}>(z+1)s_{2k}.
\]
The remaining case, $z_i=0$ for all $i\ge2k+1$ or $2k=t$, requires more effort. Let $\ib(\vz)=2j+1$ and observe that
\begin{align*}
n&=\sum_{2j+1\le i\le2k}z_i(-1)^{i-1}r_i\ge r_{2j+1}-\sum_{j+1\le i\le k-1}Z_{2i}r_{2i}-z_{2k}r_{2k}\\
&=r_{2j+1}-(r_{2j+1}-r_{2k-1})-z_{2k}r_{2k}=r_{2k-1}-z_{2k}r_{2k}.
\end{align*}
This means that $z_{2k}>z$ and gives
\[
\la un\ra_p>z_{2k}s_{2k}\ge(z+1)s_{2k}.
\]
\end{proof}

\begin{lemma}\label{L10}
Let $(a,b)$ satisfy \eqref{4} and suppose that $\la ua\ra_p>\la ub\ra_p$. The the difference $b-a$ must be of the form
\begin{equation}\label{2.9}
b-a=r_{2k}-zr_{2k+1},
\end{equation}
for some $2k+1\le t$ and $0\le z\le Z_{2k+1}-1$.
\end{lemma}
\begin{proof}
Assume, as we may, that $a$ and $b$ satisfy \eqref{2.8}. We claim that, in the present case, the inequality \eqref{4} also holds with $a$ and $b$ replaced by $c=a+p-b$ and $p$, respectively, that is, that
\begin{equation}\label{2.10}
\la uc\ra_p<\min_{c<n<p}\la un\ra_p.
\end{equation}
Indeed, write $n=m+p-b$, $a<m<b$, and observe that, by our assumption,
\[
\la un\ra_p=\la um\ra_p-\la ub\ra_p \quad\text{and}\quad \la uc\ra_p=\la ua\ra_p-\la ub\ra_p,
\]
so that
\[
\la un\ra_p-\la uc\ra_p=\la um\ra_p-\la ua\ra_p>0.
\]
Thus it remains to show that if \eqref{2.10} holds, then $b-a=p-c$ must satisfy \eqref{2.9}.

Recall that, by \eqref{2.8}, $1\le c\le p-2$. Recall also that even-indexed subsequence $r_{2j}$ starts with $r_0=p$ and terminates in either $r_t=1$ or $r_{t+1}=0$. It follows that there is an index $2k$ such that $p-r_{2k}\le c<p-r_{2k+2}$. Moreover, since $2k<t$, there is an integer $z$, $0\le z\le Z_{2k+1}-1$, such that
\begin{equation}\label{2.11}
p-r_{2k}+zr_{2k+1}\le c<p-r_{2k}+(z+1)r_{2k+1}.
\end{equation}
We will show that, in actual fact, $c=p-r_{2k}+zr_{2k+1}$, which yields \eqref{2.9}.

The rest of this proof is closely related to the arguments we used in Lemmas 8 and 9, and rests on Lemmas 2, 4-7. We begin by observing that, by \eqref{2.10} and \eqref{2.11},
\begin{equation}\label{2.12}
\la uc\ra_p<\la u(p-r_{2k}+(z+1)r_{2k+1})\ra_p=s_{2k}+(z+1)s_{2k+1}<(z+2)s_{2k+1}.
\end{equation}
Now write $c=R(\vz)$, where $\vz=(z_i)\in\CZ$. Since $\la uc\ra_p\ge z_is_i$, for all $i$, \eqref{2.12} tells us that $z_i=0$ for all $i\ge 2k+2$, so that $c=\sum_{1\le i\le2k+1}z_i(-1)^{i-1}r_i$, and that $z_{2k+1}\le z+1$. Observe also that if $z_l<Z_l$ for any specific odd index $l\le2k-1$, or if $z_l>0$ for any specific even index $l\le2k$, then
\begin{align*}
c&\le\sum_{0\le j\le k-1}Z_{2j+1}r_{2j+1}+z_{2k+1}r_{2k+1}-r_l\\
&=p-r_{2k}+z_{2k+1}r_{2k+1}-r_l<p-r_{2k}+(z_{2k+1}-1)r_{2k+1}.
\end{align*}
Since $z_{2k+1}\le z+1$, this contradicts \eqref{2.11}. Therefore, it must be that
\begin{equation}\label{2.13}
c=p-r_{2k}+z_{2k+1}r_{2k+1}=p-r_{2k}+zr_{2k+1},
\end{equation}
as claimed.
\end{proof}

\begin{lemma}\label{L11}
Let $(a,b)$ satisfy \eqref{4} and suppose that $\la ua\ra_p<\la ub\ra_p$. Then the difference $b-a$ must be of the form
\[
b-a=r_{2k-1}-zr_{2k},
\]
for some $2k\le t$ and $0\le z\le Z_{2k}-1$.
\end{lemma}
\begin{proof}
The proof of this lemma is analogous to the proof of \refL{10} and we proceed at a swifter pace. Translating the closed interval $[a,b]$ by $-a$ preserves the inequality \eqref{4} and yields
\begin{equation}\label{2.14}
\la uc\ra_p<\min_{0<n<c}\la un\ra_p \qquad [c=b-a].
\end{equation}
Recall that, by \eqref{2.8}, $2\le c\le p-2$, and that $\la ur_1\ra_p=1$. Plainly, \eqref{2.14} requires that $c$ must be $\le r_1$. We now find the index $2k-1\ge1$ such that $r_{2k+1}<c\le r_{2k-1}$, and then the integer $0\le z\le Z_{2k}-1$ such that
\begin{equation}\label{2.15}
r_{2k-1}-(z+1)r_{2k}<c\le r_{2k-1}-zr_{2k}.
\end{equation}
This is the analogue of \eqref{2.11}. And the analogue of \eqref{2.12} is
\begin{equation}\label{2.16}
\la uc\ra_p<\la u(r_{2k-1}-(z+1)r_{2k})\ra_p=s_{2k-1}+(z+1)s_{2k}<(z+2)s_{2k}.
\end{equation}
It follows that every representation $c=R(\vz)$, where $\vz=(z_i)\in\CZ$, must satisfy $z_i=0$, for all $i\ge2k+1$, as well as $z_{2k}\le z+1$. Furthermore, in order to satisfy the second inequality in \eqref{2.15}, it must be that $\vz$ is of such form that the direct evaluation of $R(\vz)$ gives 
\[
R(\vz)=z_{2l+1}-\sum_{l+1\le j\le k-1}Z_{2j}r_{2j}-z_{2k}r_{2k},
\]
for some $0\le l\le k-1$, so that $c=r_{2k-1}-z_{2k}r_{2k}$. This is the analogue of the first equality in \eqref{2.13}. Combining this with \eqref{2.15} finally gives $z_{2k}=z$, completing the proof of the lemma.
\end{proof}

\section{The Sets $S_\gD(p,q)$ and $N_\gD(p,q)$}
In Section 2, we analyzed dominant pairs of linear permutations $u\cdot n\pmod p$. In the present section  we apply this analysis to give complete characterization of the sets $S_\gD(p,q)$ and $N_\gD(p,q)$ defined in \eqref{3}. For this application, we take the parameter $u$ in Section 2 to be the multiplicative inverse of $q$ modulo $p$. Note that with $u$ so defined, the corresponding parameter $r_1$, $ur_1\equiv1\equiv uq\pmod p$, also satisfies $r_1\equiv q\pmod p$. With these conventions in place, we may now interpret \eqref{2.1} and the sequences $(r_i)$ and $(Z_i)$ as being defined in terms of the given pair $p$ and $q$. Parenthetically, for an alternative initial conditions of \eqref{2.1} we may now take $r_{-1}=q$ and $r_0=p$.

\begin{theorem}\label{T1}
\emph{(i)} Unless $q=p+1$, we have $S_\gD(p,q)=N_\gD(p,q)$. For $q=p+1$, we have $S_\gD(p,q)=N_\gD(p,q)\cup\{2\}$, and $2\notin N_\gD(p,q)$.  \\[6pt]
\emph{(ii)} We have
\begin{equation}\label{3.1}
S_\gD(p,q)=\{\, r_{i-1}-zr_i\mid0\le z\le Z_i-1\text{\ and\ }1\le i\le t\,\}.
\end{equation}
\end{theorem}

Let us recall some of the terminology introduced in Section 1. The set 
\[
S=S(p,q)=\{\,\ell_0,\ell_1,\dots,\ell_\vta\,\} \qquad[\ell_j<\ell_{j+1}],
\]
is the set of representable integers $\le\gf(p,q)$, where $\vta=\gf(p,q)/2$. Nonrepresentable integers $>0$ are
\[
N=N(p,q)=\{\,n_1,n_2,\dots,n_\vta\,\}.
\]
Thus $S_\gD=\{\ell_{j+1}-\ell_j\}$ and $N_\gD=\{n_{j+1}-n_j\}$ give all possible distances between consecutive integers of each type. Perhaps it is worth making it explicit that, for instance, the distance $n_{j+1}-n_j$ is one greater that the length of the block of consecutive representable integers between $n_j$ and $n_{j+1}$. We now introduce the characteristic function of the representable integers $\la p,q\ra$ \eqref{2.a}, 
\[
\chi(n)\coloneq\begin{cases}
1, &\text{if }n\in\la p,q\ra\\
0, &\text{otherwise}.\end{cases}
\]
Furthermore, observe that, by the Chinese remainder theorem, each integer $n$ can be written uniquely in the form
\begin{equation}\label{3.2}
n=x_nq+y_np+\dn pq,
\end{equation}
where $0\le x_n<p$, $0\le y_n<q$ and $\dn\in\Z$, so that each $n$ may be identified $n\longleftrightarrow(x_n,y_n,\dn)$ with a triple of such integers. Representable integers are $n$ with $\dn\ge0$:
\begin{equation}\label{3.3}
\chi(n)=1 \quad\text{if and only if}\quad \dn\ge0.
\end{equation}
Finally, coefficients $x_n$ and $y_n$ satisfy the congruences
\begin{equation}\label{3.4}
n\equiv x_nq\pmod p \quad\text{and}\quad n\equiv y_np\pmod q.
\end{equation}

The semigroup $\la p,q\ra$ is said to be symmetric in view of the fact that $n\in\la p,q\ra$ if and only if $pq-p-q-n\notin\la p,q\ra$. This readily follows from \eqref{3.2}. (As are the facts that $\chi(pq-p-q)=0$ and $\chi(pq-p-q+1)=1$ mentioned earlier.) This symmetry implies the identity $n_j+\ell_{\vta-j}=pq-p-q$ and shows that
\begin{equation}\label{3.5}
\{\,n_{j+1}-n_j\,\}=\{\,\ell_j-\ell_{j-1}\mid 1\le j\le\vta-1\,\}.
\end{equation}
This takes us most of the way of comparing the sets $N_\gD$ and $S_\gD$, but leaves the question of whether $\ell_\vta-\ell_{\vta-1}=\ell_j-\ell_{j-1}$ for some $j<\vta$, and thereby $S_\gD=N_\gD$.

\begin{lemma}\label{L12}
The following three assertion hold.
\begin{enumerate}
\item[\emph{(i)}] $\ell_\vta=pq-p-q+1$ and $\ell_{\vta-1}=pq-p-q-1$, so $\ell_\vta-\ell_{\vta-1}=2$.
\item[\emph{(ii)}] There exists $\ell_j$ such that $\ell_j+1=\ell_{j+1}$.
\item[\emph{(iii)}] The equation $\ell_j-\ell_{j-1}=2$ is soluble for $j<\vta$ if and only if $q>p+1$.
\end{enumerate}
\end{lemma}
\begin{proof}
We will repeatedly apply \eqref{3.2} and \eqref{3.3}. Thus, the aforementioned fact that $\chi(pq-p-q)=0$ is immediate. Now write 
\[
1=x_1q+y_1p+\gd_1pq=x_1q+y_1p-pq,
\]
and observe that, since $x_1,y_1>0$,
\[
pq-p-q+1=(x_1-1)q+(y_1-1)p\implies \chi(pq-p-q+1)=1
\]
and
\[
pq-p-q-1=(p-1-x_1)q+(q-1-y_1)p\implies\chi(pq-p-q-1)=1.
\]
This proves our first assertion. 

The second assertion is immediate if $p+1=q$. For $q\ge p+2$, we note that
\[
(p-x_1)q+1=y_1p \quad\text{and}\quad (q-y_1)p+1=x_1q,
\]
from which the assertion follows, provided that one of the terms $y_1p$ or $x_1q$ is $\le\gf(p,q)$. One readily verifies that if $q\ge p+2$, then $2\gf(p,q)\ge pq+1$, and we see that the smaller of these two terms is $<\gf(p,q)$, as required.

For the last assertion, observe that $\ell_j-\ell_{j-1}=2$ is soluble for $j<\vta$ if and only if there is $0<n<pq-p-q$ such that $\chi(n)=0$ and $\chi(n\pm1)=1$. Note that
\[
pq-p-q=(p-1)q+(q-1)p-pq,
\]
so any such $n=x_nq+y_np-pq$ must have either $x_n<p-1$ or $y_n<q-1$. Now, if $q=p+1$, then
\[
n\pm1=(x_n\pm1)q+(y_n\mp1)p-pq,
\]
and we see that either $\chi(n+1)=0$, if $x_n<p-1$, or $\chi(n-1)=0$, if $y_n<q-1$. This takes care of half of the assertion (iii).

To complete the proof, assume that $q>p+1$, write $1=x_1q+y_1p-pq$ and observe that $y_1$ must satisfy $1<y_1<q-1$. Now take $n=pq-q-2p$, so that $\chi(n)=0$ ($n\in(0,pq-p-q)$), and observe that
\[
\chi(n+1)=\chi((x_1-1)q+(y_1-2)p)=1
\]
and
\[
\chi(n-1)=\chi((p-1-x_1)q+(q-2-y_1)p)=1.
\]
This completes the proof of the lemma.
\end{proof}

\begin{proof}[Proof of Theorem 1(i)]
This is immediate from \eqref{3.5} and parts (i) and (iii) of \refL{12}.
\end{proof}

We now turn to part (ii) of \refT{1}. Observe that \eqref{3.2} with $n=1$ tells us that $x_1$ is the multiplicative inverse of $q$ modulo $p$ ($x_1=\la q^{-1}\ra_p$). We reduce the problem of understanding the set $S_\gD$ to the problem of understanding the difference set $D_\gD$ of dominant pairs $D(p,x_1)$. The key observation is contained in the following simple lemma.

\begin{lemma}\label{L13}
Put $u=x_1$. Coefficients $x_n$ satisfy the identity
\begin{equation}\label{3.55}
x_n=\la nu\ra_p\, (=\la nq^{-1}\ra_p),
\end{equation}
and
\begin{equation}\label{3.6}
\chi(n)=1\quad\text{if and only if}\quad \la nu\ra_p\le\lfloor n/q\rfloor.
\end{equation}
\end{lemma}
\begin{proof}
Formula \eqref{3.2} with an arbitrary $n$ and $n=1$ implies that $x_n$ satisfies $x_n\equiv nu\pmod p$, whence \eqref{3.55} holds. (It is also true that $y_n=\la ny_1\ra_q$, but we will have no use for this below.) Observe also that $x_nq\le n$ if and only if $\gd_n\ge0$, whence \eqref{3.6} holds by \eqref{3.3}.
\end{proof}

The next two lemmas apply \refL{13} to complete the reduction.

\begin{lemma}\label{L14}
$S_\gD(p,q)\subseteq D_\gD(p,x_1)$.
\end{lemma}
\begin{proof}
We prove the claim by showing that every pair $(a,b)=(\ell_j,\ell_{j+1})$ of consecutive elements of $S$ is a solution of \eqref{4} with $u=x_1$. To this end, fix such a pair $(a,b)$ and put $\ga=\lfloor a/q\rfloor$ and $\gb=\lfloor b/q\rfloor$. Observe that $b\le a+p$, since $\chi(a+p)=1$, whence $\gb=\ga$ or $\gb=\ga+1$. Now, by \refL{13}, for $a<n<b$, we have
\[
\la un\ra_p>\lfloor n/q\rfloor\ge\ga\ge\la ua\ra_p.
\]
So if $\la ub\ra_p$ is also $\le\ga$, inequality \eqref{4} holds  and we are done. But $\la ub\ra_p\le\ga+1$, so it remains to consider the possibility that $\la ub\ra_p=\gb=\ga+1$. Note that $\la un\ra_p\neq\la ub\ra_p$, since $0<b-n<p$, whence $\la un\ra_p\neq\ga+1$. Therefore, in fact, in this case $\la un\ra_p>\ga+1=\la ub\ra_p$, and \eqref{4} holds also.
\end{proof}

\begin{lemma}\label{L15}
$S_\gD(p,q)\supseteq D_\gD(p,x_1)$.
\end{lemma}
\begin{proof}
It is more convenient to treat trivial cases separately, and we do this first. Recall that $p\in D_\gD$, since $(0,p)$ is trivially a solution of \eqref{4}. Recall also that $p\in S_\gD$, since $\ell_0=0$ and $\ell_1=p$. The other trivial case is 1 as an element of $D_\gD$ and $S_\gD$. The first of these is due to the fact that every pair $(a,a+1)$ is a vacuous solution of \eqref{4}, and the second follows from \refL{12}(ii).

We now deal with the remaining cases. We begin by recalling that the remaining elements of $D_\gD$ are determined by solutions $(a,b)$ of \eqref{4} satisfying \eqref{2.8}. So we need to show that given any such pair $(a,b)$, there exists a pair of consecutive elements $\ell_j,\ \ell_{j+1}$ of $S$ such that $\ell_{j+1}-\ell_j=b-a$. To that end, note that, by \eqref{2.8}, $\la ua\ra_p\neq\la ub\ra_p$, and assume first that $\la ua\ra_p>\la ub\ra_p$, so that the inequality \eqref{4} now reads
\begin{equation}\label{3.7}
\la ub\ra_p<\la ua\ra_p<\min_{a<n<b}\la un\ra_p.
\end{equation}
Put $s=\la ua\ra_p$, and let $\ell=sq$ and $\ell'=sq+b-a$. Then, $\ell\equiv a\pmod p$, $\ell'\equiv b \pmod p$, and \eqref{3.7} implies that
\[
\la u\ell'\ra_p<\la u\ell\ra_p=s<\min_{\ell<n<\ell'}\la un\ra_p.
\]
But $\ell/q=s=\lfloor\ell'/q\rfloor$, and we conclude, by \refL{13}, that $\chi(\ell)=\chi(\ell')=1$ and that $\chi(n)=0$ for $\ell<n<\ell'$. Since this range of $n$ is not empty, we see that $\ell'\le\gf(p,q)$ and conclude that $\ell$ and $\ell'$ are consecutive elements of $S$. Finally, $\ell'-\ell=b-a$, as required.

We must also address the case where, in addition to satisfying \eqref{2.8}, a dominant pair $(a,b)$ is such that $\la ua\ra_p<\la ub\ra_p$. One readily verifies that an obvious adaptation of the preceding argument completes the proof of the lemma.
\end{proof}

\begin{proof}[Proof of Theorem 1(ii)]
Lemmas 14 and 15 show that $S_\gD(p,q)=D_\gD(p,u)$, where $uq\equiv1\pmod p$. The result now follows from the Main Lemma.
\end{proof}

\section{The Set $G(\Q)$}
We now return to the gaps of polynomials $\Q$. Recall that we denote the coefficients of $\Q$ by $a_m$, $\Q(x)=\sum_{0\le m\le\gf(p,q)}a_mx^m$, and that \eqref{2} hols. Using the characteristic function $\chi$ of $\la p,q\ra$, we now rewrite \eqref{2} in the form
\[
\Q(x)\equiv(1-x)\sum_n\chi(n)x^n \pmod{x^{\gf(p,q)+1}}
\]
and conclude that the coefficients $a_m$ are given by
\begin{equation}\label{4.1}
a_m=\chi(m)-\chi(m-1).
\end{equation}
In particular, it is immediate that $a_m=0$ if and only if $\chi(m)=\chi(m-1)$. Let us also note that utilizing the natural interpretation $a_m=0$ for $m<0$ or $m>\gf(p,q)$, formula \eqref{4.1} extends to all $m\in\Z$.

It is clear from \eqref{4.1} that the gaps of $\Q$ arise in two ways. We have either a string of consecutive coefficients $a_m$ of the form
\begin{enumerate}
\item[(i)] $(1,0,\dots,0,-1)=(a_m)_{m=n_j+1}^{n_{j+1}}$, with $\chi(n_j+1)=1$, so that $n_{j+1}>n_j+1$, or of the form
\item[(ii)] $(-1,0\dots,0,1)=(a_m)_{m=\ell_j+1}^{\ell_{j+1}}$, with $\chi(\ell_j+1)=0$, so that $\ell_{j+1}>\ell_j+1$.
\end{enumerate}
This description includes strings $(1,-1)$ and $(-1,1)$ with no 0s in between. The gaps associated with strings of type (i)/(ii) are positive numbers of the form $n_{j+1}-n_j-1$ or $\ell_{j+1}-\ell_j-1$, respectively. But the ``gaps of type (i)'' are always included amongst the ``gaps of type (ii)'', since, by \refT{1}(i), $N_\gD(p,q)\subseteq S_\gD(p,q)$. Therefore, using a convenient abbreviation $G(p,q)$ for $G(\Q)$, we have
\begin{equation}\label{4.2}
G(p,q)=\{\,\ell_{j+1}-\ell_j-1\mid\ell_{j+1}>\ell_j+1\,\}.
\end{equation}

\begin{theorem}\label{T2}
Let a pair of coprime integers $q\ge p\ge3$ be given, and let the sequences $(r_i)$ and $(Z_i)$ be given by \eqref{2.1}, with the initial conditions $r_0=p$ and $r_1=\la q\ra_p$ (or $r_{-1}=q$). Put
\begin{equation}\label{4.3}
\begin{aligned}
G_i=G_i(p,q) &\coloneq\{\, r_{i-1}-zr_i-1 \mid 0\le z\le Z_i-1\,\}\qquad[1\le i\le t-1], \\
G_t=G_t(p,q) &\coloneq\{\, r_{t-1}-zr_t-1 \mid 0\le z\le Z_t-2\,\}=\{\,1,2,\dots,r_{t-1}-1\,\}.
\end{aligned}
\end{equation}
We have
\begin{equation}\label{4.4}
G(p,q)=\bigcup_{i=1}^t G_i(p,q)
\end{equation}
and, noting that $Z_i=\lfloor r_{i-1}/r_i\rfloor$,
\begin{equation}\label{4.5}
\#G(p,q)=\sum_{i=1}^t\lfloor r_{i-1}/r_i\rfloor -1.
\end{equation}
\end{theorem}
\begin{proof}
This is immediate from \eqref{4.2} and \refT{1}(ii).
\end{proof}

Essentially reiterating the remark following the Main Lemma, we point out that \eqref{4.4} and \eqref{4.3} give a partition of the gap set $G(p,q)$, with the part $G_1$ containing the ``larger gaps'', $G_2$ containing the ``next larger gaps'', etc. The following two examples capture the two extremes of possible sizes of gapsets $G(p,q)$.

\begin{example}
If $r_1=\la q\ra_p=1$, so that $t=1$ and $Z_t=r_0=p$, we get
\[
G(p,q)=G_1=\{\,p-1,p-2,\dots,1\,\}.
\]
If $r_1=p-1$, so that $Z_1=1=r_2$ and $Z_2=p-1$, we get
\[
G(p,q)=G_1\cup G_2=\{p-1\}\cup\{\,p-2,p-3,\dots,1\,\}.
\]
Recall that the fact that $G(p,q)=\{1,\dots,p-1\}$ in these cases was already established in \cite{CCLMS}.
\end{example}

\begin{example}
Let $F_1=1$, $F_2=2$, \dots, be the Fibonacci numbers and take $p=F_k$ and $q=F_{k+1}$. One readily verifies that in this case \eqref{2.1} yields $t=k-1$ and
\[
Z_i\equiv1 \quad\text{and}\quad r_i=F_{k-i}\ [0\le i\le k-1].
\]
Then the sets $G_i$ are just the singletons $\{F_{k-i}-1\}$ and we get
\[
G(F_k,F_{k+1})=\{\, F_2-1,F_3-1,\dots,F_k-1\,\}.
\]
\end{example}

\begin{theorem}\label{T3}
In addition to the trivial upper bound $\#G(p,q)\le p-1$, we have the lower bound $\#G(p,q)\ge k$, where $k$ is given by the condition $F_k<p\le F_{k+1}$. Both estimates are sharp.
\end{theorem}

\begin{remark}
In view of the well-known fact that $F_k$ is the nearest integer to $\frac1{\sqrt5}(\frac{\sqrt5+1}2)^{k+1}$, we see that there is a constant $c>0$ such that $\#G(p,q)>c\log p$.
\end{remark}

In the proof of \refT{3} we will use the following simple inequality.

\begin{lemma}\label{L16}
The inequality $nF_k+F_{k-1}\le F_{k+n}$ holds for all $n,k\ge1$.
\end{lemma}
\begin{proof}
This inequality holds with equality for $n=1$ and 2. The rest follows by induction:
\[
(n+1)F_k+F_{k-1}=F_k+nF_k+F_{k-1}<F_{k+n-1}+F_{k+n}=F_{k+n+1}.
\]
\end{proof}

\begin{proof}[Proof of \refT{3}]
The upper bound is obvious since, by \refT{2}, the largest gap is $r_0-1=p-1$.

To prove the lower bound we argue by induction on the parameter $t$. For $t=1$, we have seen in Example 1 that
\[
\#G(p,q)=p-1\ge F_k\ge k.
\]
Now, the parameter $t$ corresponds to the pair $(p,q)=(r_0,r_{-1})$. To the pairs $(r_1,r_0)$ and $(r_2,r_1)$ there correspond the parameters $t-1\ (\ge1)$ and $t-2\ (\ge1)$, respectively. It is for these pairs and to the corresponding quantities $\#G(r_1,r_0)$ and $\#G(r_2,r_1)$ that our induction hypothesis will be applied.

We treat $t=2$ separately. In this case $r_2=1$ and $r_0=Z_1r_1+1$, so that $Z_1r_1\ge F_k$. Now, by \eqref{4.5}, we have
\begin{equation}\label{4.6}
\#G(p,q)=Z_1+\#G(r_1,r_0),
\end{equation}
and we need only to show that $\#G(r_1,r_0)\ge k-Z_1$. So let us consider the possibility that $\#G(r_1,r_0)<k-Z_1$. Then, by the induction hypothesis, it must be that $r_1\le F_{k-Z_1}$. But then, by \refL{16},
\[
Z_1r_1\le Z_1F_{k-Z_1}<F_k,
\]
a contradiction. Thus $\#G(p,q)\ge k$ in this case.

Since $Z_1r_1$ could be $<F_k$ if $r_2>1$, we need to work a little harder when $t>2$. We start with \eqref{4.6} and note that if $\#G(r_1,r_0)\ge k-Z_1$, then we are done. Assume that $\#G(r_1,r_0)<k-Z_1$ and write, by another application of \eqref{4.5},
\[
\#G(p,q)=Z_1+Z_2+\#G(r_2,r_1).
\]
Again, if $\#G(r_2,r_1)\ge k-Z_1-Z_2$, then we are done. So consider the possibility that $\#G(r_2,r_1)<k-Z_1-Z_2$. Our current assumption tells us, by the induction hypothesis, that $r_1\le F_{k-Z_1}$ and $r_2\le F_{k-Z_1-Z_2}$. But then, by \refL{16},
\[
r_0=Z_1r_1+r_2\le Z_1F_{k-Z_1}+F_{k-Z_1-Z_2}\le F_k,
\]
a contradiction. This completes the proof that $\#G(p,q)\ge k$.

Finally, Examples 1 and 2 show that the upper and the lower bounds of the theorem are sharp.
\end{proof}

\end{document}